\PassOptionsToPackage{unicode}{hyperref}
\PassOptionsToPackage{hyphens}{url}
\documentclass[suppldata]{interact}
\usepackage{amsmath,amssymb}
\usepackage{lmodern}
\usepackage{iftex}

\ifPDFTeX
  \usepackage[T1]{fontenc}
  \usepackage[utf8]{inputenc}
  \usepackage{textcomp} 
\else 
  \usepackage{unicode-math}
  \defaultfontfeatures{Scale=MatchLowercase}
  \defaultfontfeatures[\rmfamily]{Ligatures=TeX,Scale=1}
\fi
\IfFileExists{upquote.sty}{\usepackage{upquote}}{}
\IfFileExists{microtype.sty}{
  \usepackage[]{microtype}
  \UseMicrotypeSet[protrusion]{basicmath} 
}{}
\makeatletter
\@ifundefined{KOMAClassName}{
  \IfFileExists{parskip.sty}{%
    \usepackage{parskip}
  }{
    \setlength{\parindent}{0pt}
    \setlength{\parskip}{6pt plus 2pt minus 1pt}}
}{
  \KOMAoptions{parskip=half}}
\makeatother
\usepackage{xcolor}
\IfFileExists{xurl.sty}{\usepackage{xurl}}{} 
\IfFileExists{bookmark.sty}{\usepackage{bookmark}}{\usepackage{hyperref}}
\hypersetup{
  pdftitle={Some Notes on the Similarity of Priority Vectors Derived by the Eigenvalue Method and the Geometric Mean Method},
  hidelinks,
  pdfcreator={LaTeX via pandoc}}
\usepackage{graphicx}
\makeatletter
\def\maxwidth{\ifdim\Gin@nat@width>\linewidth\linewidth\else\Gin@nat@width\fi}
\def\maxheight{\ifdim\Gin@nat@height>\textheight\textheight\else\Gin@nat@height\fi}
\makeatother
\setkeys{Gin}{width=\maxwidth,height=\maxheight,keepaspectratio}
\makeatletter
\def\fps@figure{htbp}
\makeatother
\usepackage{amsthm}
\usepackage{todonotes}

\usepackage{cleveref}

\creflabelformat{equation}{#2#1#3}
\ifLuaTeX
  \usepackage{selnolig}  
\fi
\title{Some Notes on the Similarity of Priority Vectors Derived by the
Eigenvalue Method and the Geometric Mean Method}
\author{\name{
    Jiří Mazurek\textsuperscript{a},
    Konrad Kułakowski\textsuperscript{b}\thanks{(Corresponding author) Konrad Kułakowski. Email: konrad.kulakowski@agh.edu.pl},
    Sebastian Ernst\textsuperscript{b} and
    Michał Strada\textsuperscript{c}}
\affil{
    \textsuperscript{a}Silesian University in Opava, School of Business Administration in Karvina, Department of Informatics and Mathematics, Czech Republic; 
   \\ \textsuperscript{b}AGH University of Science and Technology, Department of Applied Computer Science, Kraków, Poland
   \\ \textsuperscript{c}Delphi Automotiv, Poland    
}}
\date{}

\begin{document}
\maketitle
\begin{abstract}
This paper examines the differences in ordinal rankings
obtained from a pairwise comparison matrix using the eigenvalue method
and the geometric mean method. First, we introduce several propositions on the (dis)similarity of both rankings concerning the matrix size and its
inconsistency expressed by the Koczkodaj's inconsistency index. Further
on, we examine the relationship between differences in both rankings and
Kendall's rank correlation coefficient \(\tau\) and Spearman's rank
coefficient \(\rho\). Apart from theoretical results, intuitive
numerical examples and Monte Carlo simulations are also provided.
\end{abstract}

\theoremstyle{plain}
\newtheorem{Theorem}{Theorem}
\newtheorem{Proposition}{Proposition}
\newtheorem{Example}{Example}
\theoremstyle{definition}
\newtheorem{Definition}{Definition}
\theoremstyle{remark}
\newtheorem{Remark}{Remark}

\section{Introduction}\label{introduction}

In many cases, comparing entities in pairs is easier and more inuitive
for experts than having to provide a complete ranking of a larger set of
items. As such, it had been used since the thirteenth century
\cite{Fiodora2011rlfa}.

However, until the early twentieth century, comparisons were only
\emph{qualitative} (e.g.~``is \(a\) better/more preferred than
\(b\)?''). The first attempts to use \emph{quantitative} comparisons
date back to the 1920s, when they were used for applications such as
comparing physical stimuli
\cite{thurstoneLawComparativeJudgment1927, thurstoneMeasurementPsychologicalValue1929}
and calculating the results of chess tournaments
\cite{zermeloBerechnungTurnierErgebnisseAls1929}. Subsequent
developments led to applications in other fields, including consumer
research, healthcare and economics.

A major breakthrough occurred in the 1970s with the development of the
Analytic Hierarchy Process (AHP) by Thomas L. Saaty
\cite{Saaty1977asmf}. Besides the theoretical background, it provided
an end-to-end solution supporting multi-criteria decision-making, and
was soon followed by a professional-grade software package, used e.g.~in
strategic planning \cite{saracogluSelectingIndustrialInvestment2013},
environmental studies
\cite{yatsaloApplicationMulticriteriaDecision2007}, risk management
\cite{najafiAnalysisEnvironmentalProjects2011}, agriculture
\cite{pazekMultiCriteriaDecisionAnalysis2010} and manufacturing
\cite{chanAHPModelSelection2010}.

Having a set of comparisons made by an expert, usually in the form of a
matrix, one may use various algorithms to transform that into a
\emph{vector of weights}. Several algorithms have been defined for that
purpose \cite{Saaty1977asmf, Crawford1987tgmp}. These weights, in turn,
can be used to order the entities to form a \emph{ranking}. For the most
commonly-used algorithms, for a given set of comparisons, these rankings
will be identical, provided the comparisons are \emph{consistent}.
However, as one can rarely expect full consistency, the resulting order
may differ. Inconsistency in pairwise comparisons is determined by the
indices \cite{Brunelli2018aaoi, Kulakowski2020iifi}. In addition to
inconsistency, the ranking can be affected by incompleteness
\cite{Kulakowski2019tqoi}. In this work, however, we will limit our
considerations to complete pairwise comparisons. We discuss the issues
of inconsistency in \ref{sec:preliminaries}.

This paper elaborates the issue of ranking similarity by providing
several propositions which specify the conditions for the rankings to be
identical and to estimate the difference of the rankings (by the means
of Kendall's and Spearman's correlation coefficients) based on the
inconsistency of the comparison matrices.

The rest of the paper is organised as follows. \ref{sec:preliminaries}
describes the most important concepts of pairwise comparisons. It is
followed by \ref{sec:ord-rank-similarity}, which provides the details
of ranking comparison methods and the aforementioned propositions, which
are the primary contribution of this work. It is followed by
\ref{sec:montecarlo}, which gives a practical context to these
contributions by means of Monte Carlo analysis. The work is summarised in
\ref{sec:conclusions}.

\section{Preliminaries}\label{sec:preliminaries}

\subsection{Multiplicative pairwise comparisons -- concepts and
notation}\label{multiplicative-pairwise-comparisons-concepts-and-notation}

This chapter provides brief preliminaries into the multiplicative
pairwise comparisons framework.

Let \(\mathcal{C}=\{c_{1},c_{2},...,c_{n}\}, n \in N, n >1\) be a
non-empty and finite set of objects (alternatives, criteria,
sub-criteria, etc.) being compared. Let \(S\) be a pairwise comparison
scale\footnote{The most popular is the fundamental scale \cite{Saaty1977asmf} where $S=\{1/9,1/8,\ldots,1/2,1,2,\ldots,8,9\}$}
\cite{Kulakowski2020utahp}. Let \(a_{ij} \subset S\), \(a_{ij} > 0\)
denote the relative importance/preference of an object \(i\) over object
\(j\), where \(i,j\in\{1,\ldots,n\}\). For instance, \(a_{ij} = 3\)
means that an object \(i\) is \emph{3 times more important}, or more
\emph{preferred}, than an object \(j\).

Pairwise comparisons of \(n\) objects form a \(n\times n\) square matrix
\(\mathbf{A}=[a_{ij}]\), called a \emph{pairwise comparison matrix} (PC
matrix, PCM):

\[\mathbf{A}=\left[\begin{array}{cccc}
1 & a_{12} & ... & a_{1n}\\
a_{21} & 1 & ... & ...\\
... & ... & 1 & ...\\
a_{n1} & ... & ... & 1
\end{array}\right],\]

A pairwise comparisons matrix \(\mathbf{A}=[a_{ij}]\) is
\emph{reciprocal} if
\[a_{ij} =1/a_{ji} \ \ \forall i,j\in\{1,\ldots,n\}.\]

Hereinafter, it is assumed that all pairwise comparison matrices are
reciprocal.

Let's call a \emph{pairwise comparison method} any decision-making
method that involves pairwise comparisons, and let a
\emph{prioritisation method} (a priority generating method) be any
procedure that derives a priority vector
\(\mathbf{w} = (w_{1},...,w_{n})\) (vector of weights of all \(n\)
compared objects) from a PC matrix.

Based on the values \(w_{1},...,w_{n}\), the compared objects can be
ranked from the best to the worst (with possible ties), which is the
goal of the majority of pairwise comparison methods.

We say that priority vector \(\mathbf{w}\) is associated with a PC
matrix \(\mathbf{A}\), or, that the priority vector \(\mathbf{w}\) is
derived by a priority generating method based on the PC matrix
\(\mathbf{A}\).

In addition, the priority vector \(\mathbf{w}\) is usually
\emph{normalized}, i.e. \[\sum_{i=1}^{n}w_{i}=1.\]

\emph{Koczkodaj's inconsistency index}, \(KI(\mathbf{A})\), of an
\(n\times n\) PC matrix \(\mathbf{A}=[a_{ij}]\) is defined as follows
\cite{Koczkodaj1993ando}:
\begin{equation}KI(\mathbf{A})=\max\left\{ 1-\min\left\{ \frac{a_{ij}}{a_{ik}a_{kj}},\frac{a_{ik}a_{kj}}{a_{ij}}\right\}| i,j,k\in\{1,\ldots,n\} \right\}\label{eq:kii}\end{equation}

It is obvious from (\ref{eq:kii}) that \(0 \le KI(\mathbf{A}) < 1\).

\subsection{The eigenvalue method and the geometric mean method for the
derivation of a priority
vector}\label{the-eigenvalue-method-and-the-geometric-mean-method-for-the-derivation-of-a-priority-vector}

The geometric mean method and the eigenvalue (eigenvector) method are
two main procedures for derivation of a priority vector from a PC matrix
\(\mathbf{A}=[a_{ij}]\). The eigenvalue (EV) method was proposed by
Saaty~\cite{Saaty1977asmf}, and the geometric mean (GM) method was
introduced by Crawford \cite{Crawford1987tgmp}.

In the EV method, the priority vector is an eigenvector corresponding to
the largest eigenvalue of \(\mathbf{A}=[a_{ij}]\) i.e.

\[\mathbf{A}\mathbf{w}=\lambda_{max}\mathbf{w},\]

where \(\lambda_{max}\geq n\) is a positive eigenvalue (the existence of
\(\lambda_{max}\) is guaranteed by the Perron-Frobenius theorem), and
\(\mathbf{w}\) is the corresponding (right) eigenvector of
\(\mathbf{A}\).

Usually, it is assumed that \(\mathbf{w}\) is normalised:
\(\Vert w \Vert = 1\).

In the GM method (the \emph{least logarithmic squares} method), the
priority vector \(\mathbf{w}\) is derived as the geometric mean of all
rows of \(\mathbf{A}\):

\[w_i=(\prod_{k=1}^{n}a_{ik})^{1/n}/\sum_{j=1}^{n}(\prod_{k=1}^{n}a_{ik})^{1/n}, \forall i.\]

This formula is equivalent to finding a solution of the following
non-linear programming problem:

\[min \sum_{i=1}^{n}\sum_{j=1}^{n} (lna_{ij} - ln\frac{w_{i}}{w_{j}})^{2}\]

\[s.t. \sum_{i=1}^{n}w_{i} = 1, w_{i}\geq 0, \forall i.\]

Again, the priority vector \(\mathbf{w}\) is normalized.

\subsection{Rank correlation
coefficients}\label{rank-correlation-coefficients}

Rank correlation measures an ordinal relationship between
\emph{rankings} of ordinal variables, where a ranking is the assignment
of ordering labels ``first'', ``second'', ``third'', etc. to different
observations of a particular variable. A \emph{rank correlation
coefficient} measures the degree of \emph{similarity} (relation) between
two rankings, and can be used to assess the significance of this
relation. The best-known rank correlation coefficients include Kendall's
rank correlation coefficient (Kendall's \(\tau\)) and Spearman's rank
correlation coefficient (Spearman's \(\rho\)).

Thereinafter, it is assumed that rankings do not contain \emph{ties},
hence they provide a \emph{total order}.

Below, definitions of Kendall's \(\tau\) and Spearman's \(\rho\) are
provided.

\begin{Definition}\label{dfn:kendall}
Let \(X\) and \(Y\) be two rankings (without ties) of \(n\) objects.
Pairs \((x_{i},x_{j})\) and \((y_{i},y_{j}), i < j\) are called
\emph{concordant}, if either \((x_{i} > x_{j} \wedge y_{i} > y_{j})\) or
\((x_{i} < x_{j} \wedge y_{i} < y_{j})\) holds (in other words, when
both objects are ranked in the same order in both rankings. Otherwise,
the pairs are called \emph{discordant}. Let \(n_{c}\) be the number of
concordant pairs and let \(n_{d}\) denote the number of discordant
pairs. Then Kendall's \(\tau\) is defined as follows:
\[\tau = \frac{n_{c}-n_{d}}{\binom{n}{2}} = \frac{2}{n(n-1)}\sum_{i<j}sgn(x_{i}-x_{j})sgn(y_{i}-y_{j})\]
\end{Definition}

From the definition it follows that \(-1 \leq \tau \leq 1\). If two
rankings are identical, then \(\tau = 1\); if they are reversed, then
\(\tau = -1\).

\begin{Definition}\label{dfn:spearman}
Let \(X\) and \(Y\) be two rankings (without ties) of \(n\) objects, and
let \(d_{i}\) be the difference between the rank of object \(i\) in
\(X\) and in \(Y\). Then, Spearman's \(\rho\) is defined as follows:
\[\rho = 1- \frac{6\sum_{i=1}^{n}d_{i}^{2}}{n(n^{2}-1)}\]
\end{Definition}

Again, \(-1 \leq \rho \leq 1\), \(\rho = 1\) for identical rankings and
\(\rho = -1\) for reversed rankings.

Both rank coefficients are applied in the next section to measure the
association between rankings of objects obtained by the eigenvalue
method and the geometric mean method with respect to the inconsistency
of the original pairwise comparison matrix.

\section{Ordinal rankings' similarity of EVM and
GMM}\label{sec:ord-rank-similarity}

From the previous example, it can be seen that priority vectors obtained
by different methods are different.

In general, priority vectors derived using EV and GM methods are
identical in the case of a consistent pairwise comparison matrix
\cite{Crawford1985anot}. Also, in the case of \(n = 3\), priority
vectors for both methods coincide even when the PC matrix is
inconsistent, see e.g. \cite{Crawford1985anot} or
\cite{Kulakowski2020utahp}. The study \cite{Herman1996amcs}
established that for not-so-inconsistent matrices, both EVM and GMM
produce very similar results, with differences `small beyond human
perception'.

For numerical comparisons of prioritisation methods see e.g.
\cite{Ishizaka2006htdp} or \cite{Lipovetzki2009coad}. Ishizaka and
Lusti \cite{Ishizaka2006htdp} concluded their study by saying:
\emph{``There is a high level of agreement between the different
(\ldots) techniques (\ldots). We do not think that one method is
superior to another. We advice decision makers also to consider other
criteria like `easy to use' in selecting of their derivation method''}.

From a theoretical point of view, Ku\l{}akowski et al.
\cite{Kulakowski2021otsb} proved the following theorem about the
distance between priority vectors obtained via EV and GM methods.

\begin{Theorem}
For every \(n\times n\) PC matrix \(\textbf{A}\) and two rankings
\(w_{EV}\) and \(w_{GM}\), \(\Vert w_{EV} \Vert\),
\(\Vert w_{GM} \Vert\), obtained by the eigenvalue and geometric mean
method respectively, it holds that:
\[\kappa^{2}-1 \leq MD (w_{EV},w_{GM}) \leq \dfrac{1}{\kappa^{2}}-1,\]
where \(\kappa = 1 - KI(A)\), \(KI\) is the Koczkodaj's inconsistency
index and \(MD\) denotes the Manhattan distance.
\end{Theorem}

\begin{proof}
See \cite{Kulakowski2021otsb}.
\end{proof}

This result provides a lower and upper bound on how ``far apart'' both
priority vectors can be. Of course, when a PC matrix is consistent
(\(KI = 0\)), then both priority vectors are identical and their
distance is zero.

However, in many situations the ordinal ranking of compared objects is
more important than the absolute values of their weights. Certainly, a
situation when a decision maker obtains a different ranking from the EVM
method than from the GMM method constitutes a serious inconvenience.

The proposition below, which directly follows from Theorem 1, addresses
this issue.

\begin{Proposition}\label{prop1}
Let \(\mathbf{A}\) be a pairwise comparison matrix with Koczkodaj's
inconsistency index \(KI(\mathbf{A})\). Let \(\mathbf{w}^{EM}\) and
\(\mathbf{w}^{GM}\) be the vectors of weights (priority vectors of
compared objects) obtained via the EV method and GM method,
respectively. Let \(O_{1}, \ldots, O_{n}\) be the ordinal ranking of all
objects \(O_{i}\) from the best (\(O_{1}\)) to the worst (\(O_{n}\))
obtained by the EV method, hence
\(w_{1}^{EV} \geq w_{2}^{EV} \geq \ldots \geq w_{n}^{EV}\). Further on,
let \(d = min \vert w_{i+1}^{EV} - w_{i}^{EV} \vert\) ,
\(i \in \lbrace 1,...,n-1\rbrace\). Then, ordinal rankings of all
objects obtained by the EV method and GM method are identical if
\begin{equation}d>\dfrac{1}{\kappa^{2}}-1\geq MD(w_{EV},w_{GM}).\label{eq:prop1eq}\end{equation}
\end{Proposition}

\begin{proof}
Let \(d = min \vert w_{i+1}^{EV} - w_{i}^{EV} \vert\). To obtain a
different priority vector \(w^{GM}\) from the GM method, it is necessary
that the (Manhattan) distance between \(w^{EV}\) and \(w^{GM}\) is
greater than \(d\) (this distance is necessary to "bridge the smallest
gap" between adjacent weights \(w_{i+1}^{EV}\) and \(w_{i+1}^{GM}\)).
However, this contradicts the assumption that
\(d > \dfrac{1}{\kappa^{2}}-1 \geq MD(w_{EV},w_{GM})\).
(The gap \(d\) is larger than the distance between \(w^{EV}\) and
\(w^{GM}\), hence it is impossible to be ``bridged''.)
\end{proof}

In real-world situations, only objects ranked at the top of the list are
important, and differences at the bottom of a ranking might be
irrelevant. An analogous proposition can be formulated for the change of
the best object.

\begin{Proposition}\label{prop2}
Let \(\mathbf{A}\) be a pairwise comparison matrix with Koczkodaj's
inconsistency index \(KI(\mathbf{A})\). Let \(\mathbf{w}^{EV}\) and
\(\mathbf{w}^{GM}\) be the vectors of weights (priority vectors of
compared objects) obtained via the EV method and GM method,
respectively. Let \(O_{1}, ..., O_{n}\) be the ordinal ranking of all
compared objects, from the best to the worst, obtained by the EV method,
hence \(w_{1}^{EV} \geq w_{2}^{EV} \geq ... \geq w_{n}^{EV}\). Further
on, let \(d^{*} = w_{1}^{EV} - w_{2}^{EV}\). Then, the object ranked
first by the EV method is also ranked first by the GM method, if
\(d^{*} > \dfrac{1}{\kappa^{2}}-1 \geq MD(w_{EV},w_{GM})\).
\end{Proposition}

\begin{proof}
It is analogous the the proof of Proposition \ref{prop1}.
\end{proof}

It is obvious that Proposition \ref{prop2} is a special case of
Proposition \ref{prop1}.

To demonstrate the use of Propositions \ref{prop1} and \ref{prop2} we
provide a following simple, yet illustrative Example \ref{ex1}.

\begin{Example}\label{ex1}
Let's consider a PC matrix \(A\) of the order \(n = 3\), where weights
derived by the EV method are:
\(w = (w_{1}, w_{2}, w_{3})=(0.60, 0.30, 0.10)\), and \(KI(A) =0.03\).
We will show (prove) that the rank of the best object does not change in
the GM method, and also the entire rankings obtained by the EV and GM
methods are identical. First, let's evaluate the best object. We have
\(d^{*} = 0.30\). Now, we evaluate
\( \dfrac{1}{\kappa^{2}}-1 = \dfrac{1}{0.97^{2}}-1 = 0.062\).
Since \(d^{*} = 0.30 > 0.062\), then according to Proposition 2, the
best object obtained by the GM method is the same. As for the entire
ordinal ranking, we get:
\(d = min \vert w_{EV}^{i+1} - w_{EV}^{i} \vert = 0.20 > 0.062\), hence,
according to Proposition 1, the ranking by GM method must be identical.
\end{Example}

\begin{Remark}\label{remark1}
Provided that the weights of all alternatives sum up to 1, it is natural
that the distance \(d\) is also limited by \(1\). Indeed, for two
alternatives \(a_1\) and \(a_2\), such that \(a_1\prec a_2\), the
greatest possible \(d = w(a_{2})-w(a_{1})\) is \(1-\epsilon\), where
\(\epsilon\) is a small number greater than \(0\). However, for three
alternatives, the highest possible \(d\) can be achieved if the
priorities of different alternatives are evenly spaced from each other.
Indeed, a simple optimisation exercise \[\begin{aligned}
&\max d\\&\,\,\,\,\,\,\,\text{s.t.}\\&d=\min\left\{ w(a_{2})-w(a_{1}),w(a_{3})-w(a_{2})\right\} \\&w(a_{1})+w(a_{2})+w(a_{3})=1\\&w(a_{1})>w(a_{2})>w(a_{3})\geq0,\end{aligned}\]
indicates that \(d = 1/3\) providing that \(w(a_1)=2/3, w(a_2)=1/3,\)
and \(w(a_3)=0\). Since all of the weights have to be positive, then in
practice we have to assume \(w(a_3)=\epsilon\) and \(d=1/3 - \epsilon\).
Thus, the maximal \(d\) is upper-bounded by \(1/3\). Similar reasoning
can be repeated for more alternatives. For example, it is easy to
observe that for \(4\) alternatives \(d=1/6-\epsilon\), and similarly
for \(5\) alternatives, \(d=1/10-\epsilon\), etc. In general, for \(n\)
alternatives, \[d=\frac{1}{1+2+\ldots+n-1}-\epsilon.\] Thus, the more
alternatives, the smaller the allowed distance \(d\), and hence the
lower the required value of the inconsistency index \(\textit{KI}\).
\end{Remark}

Next, we will provide a consequence of Proposition 1 on the values of
Kendall's rank correlation coefficient \(\tau\) and Spearman's rank
correlation coefficient \(\rho\). Both coefficients measure the
similarity of two ordinal rankings, which is applied to the rankings
obtained by the EV and GM methods.

\begin{Proposition}\label{prop3}
Let \(\mathbf{A}\) be a pairwise comparison matrix with Koczkodaj's
inconsistency index \(KI(\mathbf{A})\). Let \(\mathbf{w}^{EV}\) and
\(\mathbf{w}^{GM}\) be the vectors of weights (priority vectors of
compared objects), obtained via the EV method and GM method,
respectively. Let \(O^{EV} = (O_{1}, ..., O_{n})\) be the ordinal
ranking of all objects, from the best to the worst, obtained by the EV
method and let \(O^{GM}\) be the ordinal ranking of all objects, from
the best to the worst, obtained by the GM method. Further on, let
\(d = min \lbrace d_{i} \rbrace\) be the smallest difference between
adjacent EVM weights and let \(\dfrac{1}{\kappa^{2}}-1 = K\). Let
\(k\) be an integer number such that \(0 \leq k \leq \binom{n}{2}\),
\(k\cdot d < K\) and \((k+1) \cdot d > K\) Then
\(\tau (O^{EV}, O^{GM}) \geq \frac{\binom{n}{2} - 2k}{\binom{n}{2}}\).
\end{Proposition}

\begin{proof}
For \(k = 0\), we get \(0 < K\) and \(d > K\), hence no change in the
ranking can occur (see the proof of Proposition 1). Therefore,
\(\tau (O^{EV}, O^{GM}) \geq \frac{\binom{n}{2} - 0}{\binom{n}{2}} = 1\),
and the EVM and GMM rankings are identical. For \(k = 1\) we get
\(d < K\) and \(2d > K\); hence, at most one difference (one discordant
pair) in the rankings can occur, which means that one originally
concordant pair changes into a discordant pair and the numerator
decreases by two, therefore
\(\tau (O^{EV}, O^{GM}) \geq \frac{\binom{n}{2} - 2}{\binom{n}{2}}\).
The number \(k\) denotes the maximum number of differences in the
ranking that can occur, which is equal to the maximum number of
discordant pairs that can appear (and replace concordant pairs).
Therefore, in general,
\(\tau (O^{EV}, O^{GM}) \geq \frac{\binom{n}{2} - 2k}{\binom{n}{2}}\).
\end{proof}

Proposition \ref{prop3} postulates how much the EV and GM rankings can
differ in terms of Kendall's \(\tau\) (in other words, it postulates its
lower bound on \(\tau\)). The lower is the upper boundary (\(K\)) given
by Theorem 1 on the difference between \(w^{EV}\) and \(w^{GM}\) (in
other words, the lower the inconsistency of the original PC matrix), and
the larger the differences between adjacent weights in \(w^{EV}\), the
higher is the value of Kendall's \(\tau\).

\begin{Example}\label{ex2}
Estimate Kendall's \(\tau\) between EVM and GMM rankings if a PC matrix
\(A\) is of the order \(n = 5\), \(KI(A) = 0.11\) and \(d = 0.08\).

First, we evaluate
\(K = \dfrac{1}{\kappa^{2}}-1 = \dfrac{1}{0.89^{2}}-1 = 0.262\).
Next, we find the value of \(k\) such that
\(0 \leq k \leq \binom{n}{2}\), \(k\cdot d < K\) and
\((k+1) \cdot d > K\), which is \(k = 3\). Therefore, according to
Proposition \ref{prop3},
\(\tau \geq \frac{\binom{5}{2} - 6}{\binom{5}{2}} = 0.40\).
\end{Example}

\begin{Proposition}\label{prop4}
Let \(\mathbf{A}\) be a pairwise comparison matrix with Koczkodaj's
inconsistency index \(KI(\mathbf{A})\). Let \(\mathbf{w}^{EV}\) and
\(\mathbf{w}^{GM}\) be the vectors of weights (priority vectors of
compared objects) obtained via the EV method and GM method,
respectively. Let \(O^{EV} = (O_{1}, ..., O_{n})\) be the ordinal
ranking of all objects, from the best to the worst, obtained by the EV
method, and let \(O^{GM}\) be the ordinal ranking of all objects, from
the best to the worst, obtained by the GM method. Further on, let
\(d = min \lbrace d_{i} \rbrace\) be the smallest difference between
adjacent EVM weights and let \(\dfrac{1}{\kappa^{2}}-1 = K\). Let
\(k\) be an integer number such that \(0 \leq k \leq \binom{n}{2}\),
\(k\cdot d < K\) and \((k+1) \cdot d > K\). Then,
\(\rho(O^{EV},O^{GM}) \geq 1-\frac{6(k^{2}+k)}{n(n^{2}-1)}\).
\end{Proposition}

\begin{proof}
\emph{Proof.} Let's start with the value of \(\sum_{i-1}^{n}d_{i}^{2}\).
This value is the sum of quadratic differences in both rankings. For a
given, \(k\) -- the maximal feasible number of adjacent changes (swaps)
between both rankings -- is the maximal value of \(D\) obtained when one
object changes its rank by not more than \(k\) positions (since this
change is squared in \(D\)), while \(k\) other objects change their rank
by 1. This gives \(D_{max} = k^{2}+k\). Since actual \(D\) is smaller
than or equal to \(D_{max}\), we obtain the desired formula -- a lower
bound for Spearman's \(\rho\).
\end{proof}

In the next example, the use of Proposition \ref{prop4} is demonstrated.

\begin{Example}\label{ex3}
Estimate Spearman's \(\rho\) between EVM and GMM rankings, if a PC
matrix \(A\) is of the order \(n = 5\), \(KI(A) = 0.07\) and
\(d = 0.07\). First, we evaluate
\(K = \dfrac{1}{\kappa^{2}}-1 = \dfrac{1}{0.93^{2}}-1 = 0.1562\).
Next, we find the value of \(k\) such that
\(0 \leq k \leq \binom{n}{2}\), \(k\cdot d < K\) and
\((k+1) \cdot d > K\), which is \(k = 2\). Therefore, according to
Proposition \ref{prop4},
\(\rho (O^{EV},O^{GM}) \geq 1-\frac{6(k^{2}+k)}{n(n^{2}-1)} = 1-\frac{6(2^{2}+2)}{5(5^{2}-1)} = 0.70\).
\end{Example}

\section{Monte Carlo analysis}\label{sec:montecarlo}

The first proposition (\ref{sec:ord-rank-similarity}) formulates
sufficient conditions for the EVM and GMM rankings to be identical in an
ordinal sense. Indeed, if \(d\) (the smallest distance between
subsequent priorities in \(w_{\textit{EV}}\)) is greater than
\(1/\kappa^{2}-1\) then the orders of alternatives
determined by \(w_{\textit{EV}}\) and \(w_{\textit{GM}}\) are identical.
In practice, however, this guarantee is difficult to obtain. This is due
to two reasons. The first of them was mentioned in Remark \ref{remark1},
according to which the more alternatives, the smaller the possible
\(d\). Similarly, the more alternatives, the greater the right side of
\ref{eq:prop1eq} is. It suggests that the effect described by
Proposition \ref{prop1} can be observed with a small number of
alternatives. The second reason stems from the local nature of
Koczkodaj's index \(KI\). For the value of this index to be high, it is
enough for an expert to make a mistake in only one comparison. Hence, a
matrix with a relatively low Saaty's \(CI\) but a relatively high
Koczkodaj's \(\textit{KI}\) can be found quite often in practice.

\begin{figure}[h!]
\centering
\includegraphics[width=0.7\textwidth]{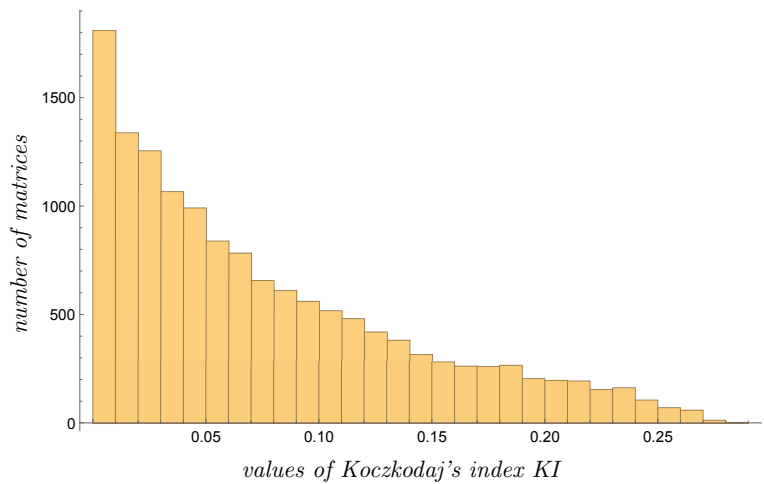}
\caption{Distribution of randomly generated PC matrices \(3\times3\) for
which condition (\ref{eq:prop1eq}) is met with respect to the value of
\(\textit{KI}\).}\label{fig:hist3x3}
\end{figure}

In the presented Monte Carlo analysis, we generate 362,750 matrices with
varying degrees of disturbance\footnote{We prepare 250 consistent random matrices and then disturb them using the disturbance coefficient from $1, 1.02, 1.04,$~up to $30$. Hence, the total number of matrices in the study is $250 * 1451 = 362,750$.}. We generate them in such a way that each
comparison in an initially connected matrix is multiplied by some scalar
\(\alpha \in [1,\beta]\), where the value \(\beta\) is gradually
increased. 

\begin{figure}[h!]
\centering
\includegraphics[width=0.7\textwidth]{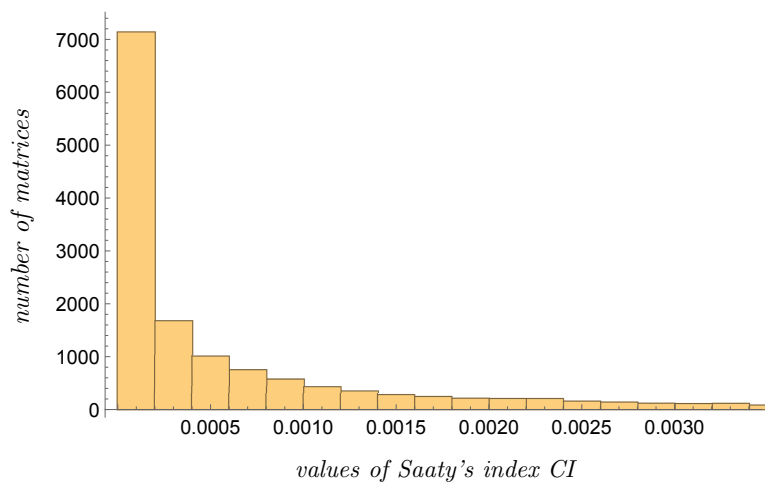}
\caption{Distribution of randomly generated PC matrices \(3\times3\) for
which condition (\ref{eq:prop1eq}) is met with respect to the value of
\(\textit{CI}\).}\label{fig:hist3x3Saaty}
\end{figure}

For these matrices, we check for how many of them the
condition in \ref{eq:prop1eq} is met, depending on the value of the
inconsistency index \(\textit{KI}\). In \ref{fig:hist3x3}, we can see
that in the case of a \(3\times 3\) matrix, the higher the inconsistency
index \(\textit{KI}\), the smaller the number of cases in which the
postulate from Proposition \ref{prop1} is satisfied. Moreover, out of
the total of 362,750 \(3\times3\) matrices, only for 14,262, i.e.~for
\(3.93\%\) of all matrices, this condition was met. These values
significantly decrease for matrices of \(4\times4\) size. With the same
number of analyzed matrices, i.e.~362,750, the sufficient criterion
(\ref{eq:prop1eq}) is met only in 636 cases (\(0.175\%\) of analyzed
cases).

\begin{figure}[h!]
\centering
\includegraphics[width=0.7\textwidth]{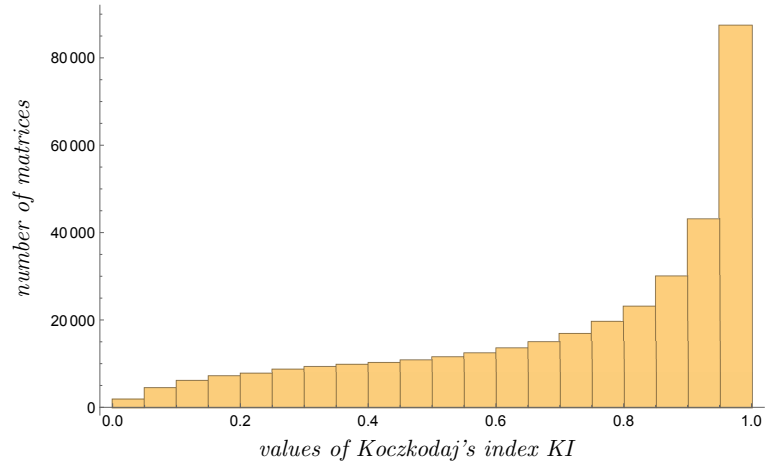}
\caption{Distribution of randomly generated PC matrices \(3\times3\) for
which condition (\ref{eq:prop1eq}) is not met with respect to the value
of \(\textit{KI}\).}\label{fig:hist3x3NOKKocz}
\end{figure}

\begin{figure}[h!]
\centering
\includegraphics[width=0.7\textwidth]{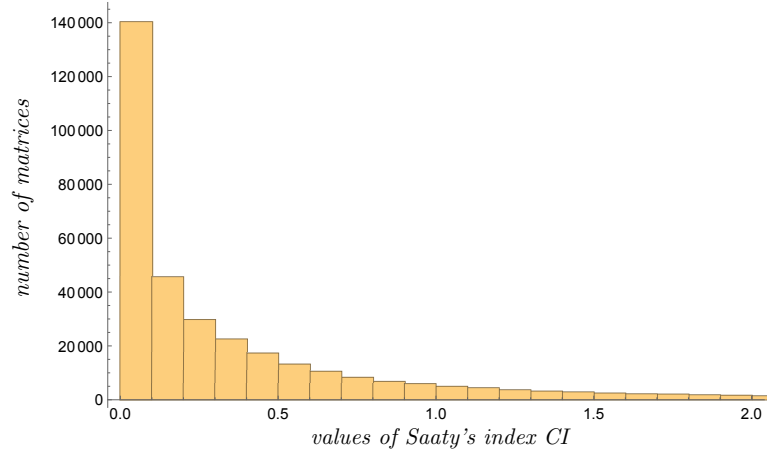}
\caption{Distribution of randomly generated PC matrices \(3\times3\) for
which condition (\ref{eq:prop1eq}) is not met with respect to the value
of \(\textit{CI}\).}\label{fig:hist3x3SaatyNOK}
\end{figure}

The values of the Saaty index for which the condition
(\ref{eq:prop1eq}) is met are low. Here, we can also see the regularity
according to which the higher the inconsistency, the fewer cases in
which the condition is met (\ref{fig:hist3x3Saaty}).

Since the proposed condition is sufficient but not necessary, there are
often matrices for which this condition is not met, and yet the order of
alternatives induced by both vectors \(w_{\textit{EV}}\) and
\(w_{\textit{GM}}\) is identical. Among these cases, the most numerous
group are matrices with a high value of the Koczkodaj's inconsistency
index (\ref{fig:hist3x3NOKKocz}).

This result is somewhat surprising, since it is pretty straightforward
to generate a matrix that is highly inconsistent in terms of
\(\textit{KI}\) -- it is enough to disturb only one comparison in the
matrix significantly. As expected, the number of PC matrices for which
the order induced by both vectors \(w_{\textit{EV}}\) and
\(w_{\textit{GM}}\) is identical decreases as the inconsistency measured
by \(\textit{CI}\) rises.

The conducted experiments show that the criterion based on Koczkodaj's
index is quite restrictive. As a result, its practical usefulness is
moderate. However, thanks to this restrictiveness, it was possible to
formulate an analytical condition for the ordinal identity of vectors
calculated using two different ranking methods. The large number of
cases not covered by the criterion (\ref{eq:prop1eq}) for which the
order induced by both ranking methods is identical may suggest that
there is a better criterion for the ordinal agreement of the vectors.
Unfortunately, at the moment, such a criterion is unknown to the authors
of this study.

\section{Conclusions}\label{sec:conclusions}

In the AHP method, the order of the alternatives is often more critical
than their numerical priorities. Therefore, in our study, we focused on
the ordinal aspect of the ranking calculated using the two methods, EVM
and GMM. We proposed a sufficient criterion for the ordinal compliance
of these two methods. We also estimated the Kendall's \(\tau\)
coefficient under certain assumptions. A Monte Carlo analysis accompanies
the theoretical considerations. It shows that the nature of the
theoretically formulated criteria is relatively restrictive. I.e. the set 
of pairwise comparisons must have a really high consistency in order to 
have guaranteed stability of the result under criterion (\ref{eq:prop1eq}). 
For example it is easy to see (Fig. \ref{fig:hist3x3}) that most $3\times 3$ matrices 
that meet the criterion have a CI smaller than $10^{-2}$. It is not much. 
On the other hand, a small (acceptable) inconsistency also does not guarantee 
that the criterion will be met (Fig. \ref{fig:hist3x3SaatyNOK}).
This restrictiveness limits the practical applicability of the
obtained academic results; on the other, it suggests the that better
criteria, applicable in more cases, may be determined. The search for
such criteria will be the subject of further research by the authors.

\section*{Acknowledgements}\label{sec:ack}

Ji\v{r}\'{i} Mazurek was supported by the Czech Grant Agency (GACR) no.
21-03085S. 


\section*{Literature}

\bibliography{papers-biblio-reviewed.bib,papers-se.bib,papers-jm.bib}

\end{document}